\documentclass[11pt]{amsart}
   \usepackage{graphicx}
\usepackage{amssymb}
\usepackage[all]{xy}
\usepackage{enumitem}
\usepackage[mathscr]{euscript}
\usepackage{hyperref}
\newtheorem{theorem}{Theorem}[section]
\newtheorem{proposition}[theorem]{Proposition}
\newtheorem{corollary}[theorem]{Corollary}
\newtheorem*{thm}{Theorem}
\newtheorem{lemma}[theorem]{Lemma}
\theoremstyle{definition}

\theoremstyle{remark}
\newtheorem{remark}[theorem]{Remark}

\newcommand{\ccd}{\mathrm{cd}}
\newcommand{\vcd}{\mathrm{vcd}}

\newcommand{\PD}{\textup{PD}}

\newcommand{\Z}{\mathbb{Z}}

\newcommand{\F}{\mathbb{F}}
\newcommand{\N}{\mathbb{N}}

\newcommand{\FP}{\mathrm{FP}}

\newcommand{\argu}{\hbox to 7truept{\hrulefill}}

\title{Subgroups of pro-$p$ $\PD^3$-groups}
\author{I.~Castellano and P.~Zalesskii}

\begin{document}
\begin{abstract} We study 3-dimensional  Poincar\'e duality pro-$p$ groups  in the spirit of the work by Robert Bieri and Jonathan
Hillmann,  and show that if such a pro-$p$ group $G$ has a nontrivial finitely presented subnormal subgroup of infinite index, then either the subgroup is cyclic and normal, or the subgroup is cyclic and the group is polycyclic, or the subgroup is Demushkin and normal in an open subgroup of $G$. 

Also, we describe the centralizers of finitely generated subgroups of  3-dimensional  Poincar\'e duality pro-$p$ groups.
\end{abstract}

\maketitle

In algebraic topology, Poincar\'e duality theorem expresses the symmetry between the homology and cohomology of closed orientable manifolds. The notion of a Poincar\'e duality group of dimension $n$ (or $\PD^n$-group, for short) originates as a purely algebraic analogue of the notion of an $n$-manifold group, that is, the fundamental group of an aspherical $n$-manifold. For dimension $n = 1 \text{ or } 2$, the modelling of $n$-manifold groups by $\PD^n$-groups is precise: the only such Poincar\'e duality groups are $\Z$ and surface groups.
In low dimension, the critical case is $n = 3$; whether every  $\PD^3$-group is a 3-manifold group is still open and represents the main problem in this area.

The notion of a duality group carries over to the realm of profinite groups but the literature has independently developed  two definitions of a profinite $\PD^n$-group $G$ at a prime $p$ whose equivalence is uncertain. The definitions differ in that one requires the profinite group $G$ to be of type $\FP_\infty$ over $\Z_p$ and the other does not.  For pro-$p$ groups the two definitions coincide and the theory becomes more amenable. The cyclic group $\Z_p$ is the only pro-$p$ $\PD^1$-group, whereas  the pro-$p$ $\PD^2$-groups  are the Demushkin groups. 

\smallskip

 In this article we focus on pro-$p$ $\PD^3$-groups and we start an investigation of their subgroups in the spirit of the work by Robert Bieri and Jonathan Hillmann  \cite{BH91,H06}, who addressed subgroup structure questions on $\PD^3$-groups motivated by considerations from 3-manifold topology (see also \cite{H16} and references there).  Typically, the existence of a subgroup satisfying some prescribed properties gives rather strong consequences for the structure of the $\PD^3$-group itself. For instance, for pro-$p$ $\PD^3$-groups we prove the following result (see Theorem~\ref{thm:subnormal}):
\begin{thm}  Let $G$ be a pro-$p$ $\PD^3$-group which has a nontrivial finitely presented subnormal subgroup $H$ of infinite index. Then one of the following holds:
\begin{enumerate}
\item $H$ is Demushkin and it is normal in an open subgroup $U$ of $G$ such that  $U/H\cong\Z_p;$ 
\item $H$ is cyclic and $G$ is virtually polycyclic; 
\item $H$ is cyclic and  normal in $G$ with $G/H$ virtually Demushkin.
\end{enumerate}
\end{thm}
The  result is the pro-$p$ analogue of Bieri and Hillman's Theorem proved in \cite{ BH91}.
The structure of the proof we provide has much in common with the classical result: we discuss all possible cases determined by the cohomological dimension of the  subnormal subgroup $H$.
But each case needs arguments specific to pro-$p$ groups. For example, some of those arguments relies on well-known results on profinite $\PD^n$-groups of~\cite{JZ,WZ}.

\smallskip

 Following the approach of  \cite{H06}, we also study centralizers of subgroups of pro-$p$ $\PD^3$-groups. We produce the  description below  which must account additional cases (especially for pro-2 groups) that do not occur in the discrete case (see Theorem~\ref{thm:centralizers}):
\begin{thm}\label{thm:centralizers}

Let $G$ be a pro-$p$ $\PD^3$-group and $H\neq1$ is a finitely generated subgroup of $G$ with $C_G(H)\neq 1$. Then one of the following holds:
\begin{enumerate}
\item $H\cong\Z_p$ and $C_G(H)/H$ is virtually Demushkin;
\item $H$ is a non-abelian Demushkin group and $C_G(H)\cong \Z_p$,;
\item $H\cong\Z_p\times\Z_p$, $C_G(H)\cong\Z_p^3$  and $G$ is virtually $\Z_p^3$;
\item $H\cong \Z_2\rtimes \Z_2$ is generalized dihedral pro-$2$ group, $C_G(H)\cong \Z_2\rtimes \Z_2$ and $G\cong \Z_2\times(\Z_2\rtimes\Z_2)$;
\item $H$ and $C_G(H)$ are free pro-$p$ groups;
\item $H$ and $C_G(H)$ are polycyclic;
\item $H\cong\Z_p$ and $C_G(H)/H$ is  virtually free pro-$p$.
\item $H$ is cyclic by virtually free and $C_G(H)\cong \Z_p$.
\end{enumerate}
\end{thm}
\subsection*{Notations} Morphisms of topological groups are assumed to be continuous and subgroups are closed. Therefore, we simply write $H \leq G$ for a closed subgroup $H$ of $G$, whereas
$K \leq_o G$ denote an open subgroup $K$ of $G$. Similarly, we use $H\trianglelefteq G$ and $K\trianglelefteq_o G$ for normal $H$ and $K$.
 
 In presence of either a topological module or a topological group, the term finitely generated (resp. presented) indicates the  pro\-perty of being topologically finitely generated (resp. presented).

\section{Preliminaries}

\subsection{Cohomological dimension of pro-$p$ groups}
For an arbitrary prime $p$ and a pro-$p$ group $G$, the {\em cohomological dimension of $G$} can be defined as
$$\ccd(G)=\sup\{n\in\N\mid H^n(G,\F_p)\neq0\}.$$
 One has the following well-known result.
\begin{proposition}[\protect{\cite[Section~7]{RZ}}]\label{prop:cd} For a pro-$p$ group $G$  one has
\begin{enumerate}
\item $\ccd(G) \leq n\in\N$ if and only if $H^{n+1}(G,\mathbb{F}_p) = 0$.
\item $\ccd(G) = n$ implies $H^n(G,M) \neq 0$
for every finite $G$-module $M$.
\item $\ccd(G)=1$ if and only if $G$ is free of rank at least 1.
\end{enumerate}
\end{proposition}
The {\em virtual cohomological dimension of $G$} is defined by
$$\vcd(G)=\inf\{\ccd(H)\mid H\leq_o G\}.$$

   \begin{proposition}[\protect{\cite{Serre}}]\label{prop:serrefree}
   For a torsion-free pro-$p$ group $G$, $\vcd(G)=\ccd(G)$. In particular, every torsion free virtually free pro-$p$ group is free pro-$p$. 
   \end{proposition}
Next we state known results on  pro-$p$ groups of finite cohomological dimension frequently used in the paper.
\begin{proposition}[\protect{\cite[Theorem~1.1]{WZ}}]\label{prop:WZ}
Let $G$ be a pro-$p$ group of finite cohomological dimension $n$ and let $N$ be a closed normal subgroup of $G$ of cohomological dimension $k$ such that the order of $H^k(N,\mathbb F_p)$ is finite.Then $\vcd(G/N) = n-k.$
\end{proposition}

We shall often used the following well-known Schur's theorem (see Theorem 8.19 \cite{K} for example).  

\begin{lemma}\label{shur} Let $G$ be a group having  center of finite index. The commutator $[G,G]$ is finite.

\end{lemma}

The following is the pro-$p$ analogue of \cite[Theorem~8.8]{Bieri}.
\begin{corollary}\label{cor:center}
Let $G$ be a  pro-$p$ group of finite cohomological dimension $n$. If the center $\mathcal{Z}(G)$ has cohomological dimension  $n$, then $G$ is abelian.
\end{corollary}
\begin{proof} Since $\mathcal{Z}(G)$ is a torsion-free abelian pro-$p$ group, $\mathcal{Z}(G)$ is a free abelian pro-$p$ group of rank $= n$ (see \cite[Theorem~4.3.4]{RZ}). By Proposition~\ref{prop:WZ}, $\mathcal{Z}(G)$ has finite index in $G$ and so the commutator $\overline{[G,G]}$ is finite by Lemma \ref{shur}. Since $G$ is torsion-free,  $\overline{[G,G]}=1$ and $G$ is abelian.
\end{proof}
\subsection{Pro-$p$ $	\PD^n$-groups}
The pro-$p$ group $G$ is called a {\em pro-$p$ Poincar\'e duality group of dimension $n$} (or pro-$p$ $\PD^n$-group, for short) if the following properties are satisfied:
\begin{enumerate}
\item[(D1)] $\ccd(G)=n$,
\item[(D2)] $|H^k(G,\F_p)|$ is finite for all $k\leq n$,
\item[(D3)] $H^k(G,\F_p[[G]])=0$ for $1\leq k\leq n-1$, and
\item[(D4)] $H^n(G,\F_p[[G]])\cong\F_p.$
\end{enumerate}
All pro-$p$ $\PD^1$-groups are  infinite cyclic and all pro-$p$ $\PD^2$-groups  are Demushkin.
The Demushkin groups $D$ are one-relator pro-$p$ groups of cohomological  dimension 2 such that $H^2(D,\F_p)\cong\F_p$ and the cup-product
$$\cup\colon H^1(D,\F_p)\times H^1(D,\F_p)\to H^2(D,\F_p)$$
is a non-singular bilinear form.

We end the preliminaries by collecting several well-known results on pro-$p$ $\PD^n$-groups that will be used further on.
\begin{proposition}[\protect{\cite[Proposition~4.4.1]{SW}}]\label{prop:open}
Let $G$ be a pro-$p$ group of finite cohomological dimension $n$ and $U$  an open subgroup. Then $G$ is a pro-$p$ $\PD^n$-group if and only if $U$ is a pro-$p$ $\PD^n$-group.
\end{proposition}
Note that the previous result is stated for pro-$p$ groups but it holds in general for profinite groups.
\begin{proposition}[\protect{\cite[I.\S4, Exercise~5(b)]{Serre}}]\label{prop:serre}
A subgroup of infinite index in a pro-$p$ Poincar\'e duality group of dimension $n$ has cohomological dimension at most $n-1$.
\end{proposition}
\begin{proposition}[\protect{\cite[Corollary~1.5]{JZ}}]\label{JZ}
Let $G$ be a 
pro-$p$ $\PD^3$-group and $N$ a non-trivial finitely generated
normal subgroup of $G$ of infinite index. Then either $N$ is infinite cyclic and $G/N$ is virtually Demushkin
or $N$ is Demushkin and $G/N$ is virtually infinite cyclic.
\end{proposition}

We shall need also the following

\begin{proposition}\label{prop:norm} Let $H\neq1$ be a finitely presented subgroup of a pro-$p$ $\PD^3$-group $G$. If $\ccd(H)=\ccd(N_G(H))$, then $H$ is open in $N_G(H)$.
\end{proposition}
\begin{proof}
For $\ccd(H)=3$ the result follows by Proposition~\ref{prop:serre}. If $\ccd(H)<3$, then $H$ is open in $N_G(H)$ by Proposition~\ref{prop:WZ}. 
\end{proof}

\section{Polycyclic pro-$p$ groups}
A pro-$p$ group $G$ is called {\em polycyclic} if there is a finite series of closed normal subgroups of $G$ 
\begin{equation}\label{eq:series}
1=G_0\trianglelefteq G_1\trianglelefteq\cdots \trianglelefteq G_n=G
\end{equation}
such that each factor group $G_i/G_{i-1}$ is cyclic for $i=1,\ldots,n$. 

Let $G$ be a polycyclic pro-$p$ group. The {\em Hirsch length} of $G$, denoted by  $h(G)$, is the number of factors $G_i/G_{i-1}$ in the series \eqref{eq:series} that are isomorphic to $\Z_p$. Clearly, this number is independent on the choice of the series.

Many results on abstract polycyclic groups find their analogue in the  pro-$p$ world and the proofs often carry over up to minimal adjustments (see \cite{Segal} for the abstract case). 
\begin{lemma}\label{lem:poly} Let $G$ be a pro-$p$ group, $H$  a finitely generated normal subgroup of $G$ and $K\leq_o H \trianglelefteq G$. Then the normal core $K^G$ is open in $K$.
\end{lemma}
\begin{proof} Replacing $K$ by its normal core $K^H$ in $H$, we may assume $K\trianglelefteq_o H$. If $g\in G$, then $K^g\trianglelefteq_o H$ and $H/K^g\cong H/K$. There are finitely many morphisms of $H$ onto $H/K$, since there are at most $|H/K|$ possible images for each of the finitely many generators of $H$. Therefore, there are finitely many distinct groups among the $K^g$ as $g$ runs through $G$; call them $K_1,\ldots,K_n$. Clearly, $K^G=K_1\cap\ldots\cap K_n$ is open in $K$ and normal in $G$.
\end{proof}
\begin{remark}\label{rem:norm} The proof above also shows that the normalizer $N_G(K)$ has finite index in $G$ whenever $H$ is  finitely generated  and $K\trianglelefteq_o H \trianglelefteq G$ for some pro-$p$ group $G$. 
\end{remark}
\begin{proposition}\label{prop:poly}
Every polycyclic pro-$p$ group $G$ with $h(G)\geq 1$ is  poly-$\Z_p$ by finite.
\end{proposition}
\begin{proof}
Suppose that $G$ has a finite series of subgroups $1=G_0\trianglelefteq G_1\trianglelefteq\cdots\trianglelefteq G_n=G$ with cyclic  factors $G_i/G_{i-1}$. We need to prove that $G$ has a poly-$\Z_p$ subgroup that is normal and open. If $n=1$ there is nothing to prove. For $n>1$ we proceed by induction. Suppose that $G_{n-1}$ has a poly-$\Z_p$ subgroup $N$ which is normal and open. If $G/G_{n-1}$ is finite, then the core $N^G$ is a poly-$\Z_p$ subgroup of $G$ that is normal and open.  If $G/G_{n-1}$ is infinite, we apply Lemma~\ref{lem:poly} with $N$ for $K$ and $G_{n-1}$ for $H$. There exists an open subgroup $U\trianglelefteq_o N\trianglelefteq_o G_{n-1}$ such that $U\trianglelefteq G$. In particular, $G/U$ is virtually $\Z_p$, i.e.,  there is a finite index subgroup $K/U$ of $G/U$  such that $K/U\cong\Z_p$.  Thus, $K$  is poly-$\Z_p$ because $U$ is poly-$\Z_p$ and $K/U\cong\Z_p$. 
\end{proof}
By a combination of the last paragraph of \cite{camina} and \cite[Theorem~1.3]{WilZ}, one obtain the following list of isomorphism types for torsion-free virtually abelian pro-$p$ groups.
\begin{proposition}\label{prop:polycyclic_abelian}
Every virtually abelian  pro-$p$ group  of cohomological dimension 3 has one of the following isomorphism types:
\begin{enumerate}
\item[(i)] for $p>3$, $\Z_p\times\Z_p\times\Z_p$;
\item[(ii)] for $p=3$, in addition to $\Z_3\times\Z_3\times\Z_3$, one has a torsion-free extension of $\Z_3\times\Z_3\times\Z_3$ by $C_3$;
\item[(iii)] for $p=2$, in addition to $\Z_2\times\Z_2\times\Z_2$, one has  a torsion-free extension of $\Z_2 \times \Z_2 \times \Z_2$ by one of the following finite 2-groups: $C_2, C_4, C_8, D_2, D_4, D_8, Q_{16}$.
\end{enumerate}
\end{proposition}
\begin{proposition}\label{prop:virtuallyZp3}
Let $G$ be a pro-$p$ $\PD^3$-group such that $\mathcal{Z}(G)$ is not cyclic. If  $p>2$, then $G$ is abelian. If $p=2$, then $G\cong \Z_2\times(\Z_2\rtimes\Z_2)$.
\end{proposition}
\begin{proof} Finitely generated torsion-free abelian pro-$p$ groups are free abelian pro-$p$ of finite rank, then the rank of $\mathcal{Z}(G)$   equals either 2 or 3. 

 If $\mathcal{Z}(G)\cong\Z_p^3$, Proposition~\ref{prop:serre} implies that $\mathcal{Z}(G)$ has finite index in $G$ and so the commutator is finite by Lemma \ref{shur}. For $G$ is torsion-free, $G$ is abelian.

Suppose $\mathcal{Z}(G)\cong\Z_p\times\Z_p$. By Proposition~\ref{JZ}, the quotient $G/\mathcal{Z}(G)$ is virtually $\Z_p$, i.e. $G/\mathcal{Z}(G)\cong \Z_p$  for $p>2$ and either $\Z_2$ or $\Z_2\rtimes C_2$ for $p=2$ (since $Aut(\Z_p)\cong \Z_p\times C_{p-1}$ or $\Z_2\times C_2$). But center by cyclic group is abelian so, for  $p=2$, $G/\mathcal{Z}(G)\cong\Z_2\rtimes C_2$ is infinite dihedral. Thus $G$ is an extension of $\Z_2\times\Z_2$ by $\Z_2\rtimes C_2$ and easy computations yield $G\cong \Z_2\times(\Z_2\rtimes\Z_2)$. 

Suppose now $Z(G)$ is free abelian of rank 3 and so is of finite index. By Lemma \ref{shur}, $[G,G]$ is finite and  $G$ is abelian since $G$ is torsion free. This finishes the proof.
\end{proof}
We finish the section with the description of torsion-free polycyclic subgroups of Hirsch length 3. They are exactly polycyclic $\PD^3$ pro-$p$ groups. 
\begin{theorem} \label{polycyclic}
Let $G$ be a torsion-free  polycyclic pro-$p$ group of Hirsch length 3. Then one of the following holds:
\begin{enumerate}
\item[(i)] $G$ is virtually abelian and so described in Proposition \ref{prop:polycyclic_abelian};
\item[(ii)]
 $G\cong (\Z_p\times \Z_p)\rtimes \Z_p$ for $p>2$ and has a subgroup of index 2 isomorphic to $(\Z_2\times \Z_2)\rtimes \Z_2$  if $p=2$.
 \end{enumerate}
\end{theorem}
\begin{proof}  Let $A$ be a maximal normal abelian subgroup of $G$. Then $C_G(A)$ is the kernel of the homomorphism $G\longrightarrow Aut(A)$. If  $A$ is cyclic, then  $Aut(A)$ is virtually cyclic; therefore  $C_G(A)$  has Hirsch length $\geq 2$ and  so is virtually abelian of rank 2, a contradiction.

If $rank(A)=3$, we are in case (i). 
Suppose $rank(A)=2$. Then $G/A$ is virtually cyclic. Therefore, for $p>2$, $G/A\cong \Z_p$ and $G\cong A\rtimes \Z_p$. For $p=2$, the quotient group $G/A$ can also be infinite dihedral and so contains an infinite cyclic subgroup of index 2 whose inverse image in $G$ has the required structure. This finishes the proof.
\end{proof} 
\begin{corollary} 
Let $G$ be a torsion-free polycyclic pro-$p$ group of Hirsch length 3. If $p>3$ then $G=\Z_p\times \Z_p\rtimes \Z_p$.  
\end{corollary}

\section{Centralizers  in pro-$p$ $\PD^3$ groups}
\begin{theorem}\label{thm:centralizers}
Let $G$ be a pro-$p$ $\PD^3$-group and $H\neq1$ is a finitely generated subgroup of $G$ with $C_G(H)\neq 1$. Then one of the following holds:
\begin{enumerate}
\item $H\cong\Z_p$ and $C_G(H)/H$ is virtually Demushkin;
\item $H$ is a non-abelian Demushkin group and $C_G(H)\cong \Z_p$,;
\item $H\cong\Z_p\times\Z_p$, $C_G(H)\cong\Z_p^3$  and $G$ is virtually $\Z_p^3$;
\item $H\cong \Z_2\rtimes \Z_2$ is generalized dihedral pro-$2$ group, $C_G(H)\cong \Z_2\rtimes \Z_2$ and $G\cong \Z_2\times(\Z_2\rtimes\Z_2)$;
\item $H$ and $C_G(H)$ are non-abelian free pro-$p$ groups;
\item $H$ and $C_G(H)$ are polycyclic;
\item $H\cong\Z_p$ and $C_G(H)/H$ is  virtually free pro-$p$.
\item $H$ is cyclic by virtually free and $C_G(H)\cong \Z_p$.
\item $C_G(H)$ is cyclic, $H$ is open and $G$ is cyclic by virtually Demushkin.
\end{enumerate}
\end{theorem}
\begin{proof} We distinguish three cases each of which has several subcases.
\begin{description}[style=unboxed,leftmargin=0cm]
\item[Case 1] Suppose $HC_G(H)\trianglelefteq_o G$. By Proposition~\ref{prop:open},  $HC_G(H)$ is a pro-$p$ $\PD^3$-group satisfying the conclusions of Proposition~\ref{JZ} for $H$ finitely generated and normal. Then either $H\cong\Z_p$ or $H$ is Demushkin or $H$ is open $\PD^3$. For $H\cong\Z_p$, $C_G(H)=HC_G(H)$ and $C_G(H)/H$ is virtually Demushkin by Proposition~\ref{JZ} again (i.e. (1) holds). If $H$ is Demushkin with trivial center, then $HC_G(H)\cong H\times C_G(H)$ and $C_G(H)\cong \Z_p$  (see Proposition~\ref{prop:serrefree}) and so (2) holds. Demushkin groups with non-trivial center are isomorphic to $\Z_p\times\Z_p$ or to the Klein bottle pro-$2$ group $\Z_2\rtimes \Z_2$. For $H\cong \Z_p\times\Z_p$, $C_G(H)=HC_G(H)$ and for $p>2$ is central extension of $H$ by $\Z_p$ and so $C_G(H)\cong \Z_p^3$, i.e. (3) holds. If $H\cong \Z_2\times \Z_2$ or $H\cong\Z_2\rtimes \Z_2$, then $C_G(H)$ contains $\Z_2\times \Z_2$ and so, by Theorem \ref{polycyclic}, $G$ is virtually $\Z_2\times(\Z_2\rtimes\Z_2)$. Thus  (4) holds. 

Finally, if $H$ is open in $G$ then the center $\mathcal{Z}(H^G)$ of the normal core of $H$ is non-trivial and normal in $G$. Then we have either  (6) or (9) by Proposition~\ref{JZ}.
\end{description}
\begin{description}[style=unboxed,leftmargin=0cm]
\item[Case 2] Suppose $\ccd(HC_G(H))=2$. If $\mathcal{Z}(H)=1$, then $HC_G(H)\cong H\times C_G(H)$ and both $H$ and $C_G(H)$ are free pro-$p$ so (5) holds. 

\smallskip
Assume $\mathcal{Z}(H)$ is non-trivial. Then by Corollary \ref{cor:center} either $HC_G(H)$ is abelian  (hence $HC_G(H)=C_G(H)$ and we have case (6)) or $\mathcal{Z}(H)\cong \Z_p$.
 In the latter case, by Proposition~\ref{prop:WZ}, the quotient $HC_G(H)/\mathcal{Z}(H)$ is virtually free and $H/\mathcal{Z}(H)$ is either open in $HC_G(H)/\mathcal{Z}(H)$ or finite. So either $\mathcal{Z}(H)\trianglelefteq_o H$ or $H\trianglelefteq_o HC_G(H)$.

In the first case $H=\mathcal{Z}(H)\cong\Z_p$.  So $HC_G(H)=C_G(H)$ which is $\Z_p$ by virtually free by Proposition~\ref{prop:WZ}, i.e. (7) holds.

In the second case ($H\trianglelefteq_o HC_G(H)$) we have $\mathcal{Z}(H)\trianglelefteq_o C_G(H)$ and so $C_G(H)\cong\Z_p\cong\mathcal{Z}(H)$. Then  either $H$ is polycyclic and so is $H C_G(H)$ (i.e. (6) holds)   or $H/\mathcal{Z}(H)$ is virtually free non-abelian  and so (8) holds.

\end{description}
\begin{description}[style=unboxed,leftmargin=0cm]
\item[Case 3] \noindent Suppose $HC_G(H)$ is free. Then $H$ and $C_G(H)$ are free pro-$p$. But if $H$ is free pro-$p$ non-abelian then $C_G(H)=1$ contradicting the hypothesis. So   $H=C_G(H)\cong \Z_p$, i.e. case (6) holds.  
\end{description}
\end{proof}

\begin{remark} It is well-known that discrete $\PD^3$-groups do not contains products of nonabelian free groups \cite{Peter}. At this stage, we do not know whether case (5) can occur in the pro-$p$ context.

In cases (1), (7) and (8) the finite subgroups of the mentioned virtually Demushkin and virtually free pro-$p$ group are actually cyclic, since the inverse images of them in $G$ are torsion-free virtually cyclic and therefore are cyclic. Moreover, $C_G(H)$ in case (7) and $H$ in case (8) are the fundamental groups of finite graphs of infinite cyclic groups by \cite{HZ13}. It follows that, in this cases, $C_G(H)$ and $H$ are the pro-$p$ completions of abstract fundamental groups of finite graph of cyclic groups.

A natural example of $C_G(H)$ in case (1) is the pro-$p$ completion of a residually-$p$ fundamental group of a Seifert 3-manifold.
\end{remark}
\section{Subnormal subgroups of pro-$p$ $\PD^3$-groups}
\begin{lemma}\label{lem:unique} Let $G$ be a pro-p group and $N$ an infinite maximal cyclic normal subgroup of $G$. If   $G/N$  does not have  non-trivial normal cyclic subgroups, then  $N$ is characteristic.
\end{lemma}
\begin{proof} Suppose $G$ has a normal subgroup $K\cong\Z_p$. The projection of $K$ to $G/N$ is normal and cyclic, and then trivial. Hence all normal cyclic subgroups of $G$ are contained in $N$. In particular, $\phi(N)\leq N$ for every automorphism $\phi$ of $G$.
\end{proof}
\begin{theorem} \label{thm:subnormal} Let $G$ be a pro-$p$ $\PD^3$-group which has a nontrivial finitely presented subgroup $H$ which is subnormal and of infinite index in $G$. Then one of the following holds:
\begin{enumerate}
\item $H$ is Demushkin and it is normal in an open subgroup $U$ of $G$ such that  $U/H\cong\Z_p$;
\item $H$ is cyclic and $G$ is  polycyclic;
\item $H$ is cyclic and  normal in $G$ with $G/H$ virtually Demushkin.
\end{enumerate}
\end{theorem}
\begin{proof} 
Let $\{J_i\mid 0\leq i\leq n\}$ be a chain  of minimal length $n$ among subnormal chains in $G$ with $H=J_0$ and let $k=\min\{i=1,\ldots,n\mid [J_i\colon H]=\infty\}$. Since $H$ has infinite index in $G$, it has cohomological dimension at most $2$ by Proposition~\ref{prop:serre}. Therefore,  $H$ is either infinite cyclic, a nonabelian free pro-$p$ group or it has cohomological dimension $2$. We therefore distinguish three cases some of which have several subcases.
\begin{description}[style=unboxed,leftmargin=0cm]
\item[Case 1] \noindent Assume $\ccd(H)=2$. Since $H$ is finitely presented and $H$ is open in $J_{k-1}$, $J_{k-1}$ is finitely presented and so the order of $H^2(J_{k-1},\F_p)$ is finite. Proposition~\ref{prop:WZ} yields $\ccd(J_k)=3$  because $J_{k-1}$ has infinite index in $J_k$.  By Proposition~\ref{prop:serre}, $J_k$ is open in $G$ and  $J_k$ is a pro-$p$ $\PD^3$-group (see Proposition~\ref{prop:open}). Proposition~\ref{JZ} implies that $J_{k-1}$ is Demushkin and $J_k/J_{k-1}$ is virtually cyclic. Finally, since $H$ is open in $J_{k-1}$, $H$ is Demushkin and, by Remark \ref{rem:norm}, $N_{J_k}(H)$ is open in $J_k$. Hence  (1) holds in this case. 
\end{description}

\medskip
For the cases with $\ccd(H)=1$ we may assume $k=1$ by Proposition~\ref{prop:serrefree}.
 Moreover, Proposition~\ref{prop:norm} yields $\ccd(J_1)> 1$ and so $J_1$ is not free pro-$p$. By Schreier index formula, there exist a maximal subgroup among finitely generated normal free pro-$p$ subgroups of $J_1$ containing $H$ as an open subgroup, so
from now on, we assume $H$ is maximal.
\begin{description}[style=unboxed,leftmargin=0cm]
\item[Case 2] \noindent Suppose $H$ is non-abelian.  First we note that  $\ccd(J_1)\neq 3$  because   
$H$ is not cyclic  (see Propositions~\ref{prop:open},~\ref{prop:serre} and~\ref{JZ}). Thus, $\ccd(J_1)=2$ and  $J_1/H$ is virtually free pro-$p$ by Proposition~\ref{prop:WZ}.

  Let $g\in G$ be an element that normalizes $J_1$. Then $HH^g/H\cong H^g/H\cap H^g$ is a finitely generated normal subgroup of the virtually free pro-$p$ group $J_1/H$ and so
    either $HH^g/H$ is finite or it has finite index in $J_1/H$. 
\end{description}
\begin{description}[style=unboxed,leftmargin=0cm]
\item[Subcase 2a]  $HH^g/H$ is finite for every $g\in J_2$. Then $H$ is open in $HH^g$ and by the maximality of $H$, $HH^g=H$ for every $g\in J_2$. Then, $H$ is normal in $J_2$ contradicting our assumption that the length $n$ was minimal. Thus this subcase does not occur.

\item[Subcase 2b]  $HH^g/H$ has finite index in $J_1/H$ for some $g\in J_2$. It follows that  $HH^g$ has finite index in $J_1$ and $J_1$ is finitely generated. Then $J_1/H$ is finitely generated, and so $J_1$ is finitely presented since it is virtually the semidirect product of finitely generated free pro-$p$ groups. By arguing as in Case 1 with $J_1$ for $H$, we deduce that $J_1$ is a pro-$p$ $\PD^2$, i.e., $J_1$ is Demushkin.  The group $J_1$ can not be  polycyclic since it contains the non-abelian free pro-$p$ subgroup $H$. If $J_1$ is Demushkin but not polycyclic this contradicts the fact that a non-soluble Demushkin group does not possess finitely generated normal free subgroups  (see \cite[Theorem 3]{HS} for example). Thus this subcase can not occur either, i.e. Case 2 does not occur.
\end{description}
\begin{description}[style=unboxed,leftmargin=0cm]
\item[Case 3] \noindent  Assume $H\cong\Z_p$.

\medskip 
{\bf Claim.} If  $H$ is not normal in $G$, then $G$ is  polycyclic.

\medskip
Proof. 
  If  there is $g\in J_2$ such that $HH^g$ is  2-generated, then $cd(HH^g)=2$ and  so $J_{1}/(HH^g)$ is virtually free. If $J_{1}/(HH^g)$ is infinite, it follows that $J_{1}$ contains a subgroup $HH^g\rtimes C$, with $C$  infinite cyclic, whose cohomological dimension is 3 and, hence, it is open in $G$ as needed. Otherwise, $HH^g$ is open in $J_{1}$ and so $cd(J_{1})=2$. As in Subcase 2b, $J_1$ is finitely presented. Then $J_2/J_{1}$ is virtually free and, as above, $J_2$ contains $J_{1}\rtimes C$ that has cohomological dimension 3. So $J_2$ is open in $G$ and therefore, by Proposition \ref{JZ}, $J_2/J_1$ is  virtually cyclic. Since $HH^g$ is polycyclic, we deduce that $G$ is polycyclic.  

Assume now that $HH^g$ is cyclic for every $g\in J_2$. Then $[HH^g:H]$ is finite and let $S=\bigcup S_j$ be an ascending union of finite subsets of $J_2$ such that $S$ is dense in $J_2$. Then $$A=\overline{\bigcup_{j}\prod_{g\in S_j} H^g},$$ (where $\prod$ means the internal product of normal subgroups in $J_1$)  is an abelian normal subgroup of $J_{2}$ that has to be of rank at most 3. If $A$ is not cyclic then as before $A\cong \Z_p\times \Z_p$ and $J_2/A$ is virtually free;  so $J_2$ contains $A\rtimes C$ with $C$ infinite cyclic, so $A\rtimes C$ and therefore $J_2$ is open in $G$ and hence $G$ is  polycyclic. If $A$ is cyclic then $H$ is open in $A$ and so is normal in $J_2$, a contradiction. This proves   the claim.    

\medskip
Thus from now on we assume that $H$ is normal in $G$. 
 By Proposition~\ref{JZ}, $G/H$ is virtually  Demushkin that corresponds to Case (3).
 \end{description}
\end{proof}


\end{document}